\documentclass[reqno]{amsart}
\usepackage{amssymb}
\usepackage{amsthm}
\usepackage{amsmath}
\usepackage{enumitem}
\usepackage{xcolor}
\usepackage{booktabs}

\usepackage[hyphens]{url}
\usepackage{hyperref}
\usepackage[hyphenbreaks]{breakurl}

\usepackage{amsaddr}

\hypersetup{
    colorlinks=true,
    linkcolor=babyblue,
    filecolor=babyblue,      
    urlcolor=babyblue,
    citecolor=babyblue
    }

\definecolor{babyblue}{rgb}{0.1, 0.6, 0.75}

\definecolor{orcidlogocol}{rgb}{0.1, 0.6, 0.75}

\newtheorem{theorem}{Theorem}[section]
\newtheorem{corollary}[theorem]{Corollary}
\newtheorem{proposition}[theorem]{Proposition}
\newtheorem{lemma}[theorem]{Lemma}

\theoremstyle{definition}

\newtheorem{remark}[theorem]{Remark}

\def\Aut{{\rm Aut}}

\def\Sym{{\rm Sym}}

\def\ker{{\rm ker}\,}
\def\ord{{\rm ord}}

\def\Skew{{\rm Skew}}

\newcommand{\NN}{\mathbb N}
\newcommand{\ZZ}{\mathbb Z}

\def\CC{{C\nolinebreak[4]\hspace{-.05em}\raisebox{.4ex}{\tiny\bf ++}}}

\begin{document}
\title{Quotients of skew morphisms of cyclic groups}
\date{\today}
\author[Martin Bachrat{\' y}]{Martin Bachrat{\' y}}
\address{Faculty of Civil Engineering, Slovak University of Technology, Bratislava 81005, Slovakia}
\email{martin.bachraty@stuba.sk}
\thanks{The author acknowledges the use of the {\sc Magma} system~\cite{Magma} to find examples of skew morphisms relevant to this paper. The author also acknowledges support from the APVV Research Grants 17-0428 and 19-0308, and the VEGA Research Grants 1/0206/20 and 1/0567/22.}

\begin{abstract}
A skew morphism of a finite group $B$ is a permutation $\varphi$ of $B$ that preserves the identity element of $B$ and has the property that for every $a\in B$ there exists a positive integer $i_a$ such that $\varphi(ab) = \varphi(a)\varphi^{i_a}(b)$ for all $b\in B$. The problem of classifying skew morphisms for all finite cyclic groups is notoriously hard, with no such classification available up to date. Each skew morphism $\varphi$ of $\ZZ_n$ is closely related to a specific skew morphism of $\ZZ_{|\!\langle \varphi \rangle\!|}$, called the quotient of $\varphi$. In this paper, we use this relationship and other observations to prove new theorems about skew morphisms of finite cyclic groups. In particular, we classify skew morphisms for all cyclic groups of order $2^em$ with $e\in \{0,1,2,3,4\}$ and $m$ odd and square-free. We also develop an algorithm for finding skew morphisms of cyclic groups, and implement this algorithm in {\sc Magma} to obtain a census of all skew morphisms for cyclic groups of order up to $161$.

During the preparation of this paper we noticed a few flaws in Section~5 of the paper \emph{Cyclic complements and skew morphisms of groups} [J.\ Algebra\/ {\bf 453} (2016), 68--100]. We propose and prove weaker versions of the problematic original assertions (namely Lemma~5.3(b), Theorem~5.6 and Corollary~5.7), and show that our modifications can be used to fix all consequent proofs (in the aforementioned paper) that use at least one of those problematic assertions. \\[0.4em]
\textit{Math.\ Subj.\ Class.\ (2020)\,: 20B25, 05C25, 05E18}
\end{abstract}

\maketitle

\section{Introduction}\label{sec:intro}

A skew morphism of a finite group $B$ is a permutation $\varphi$ of $B$ that preserves the identity element of $B$ and has the property that for every $a\in B$ there exists a positive integer $i_a$ such that $\varphi(ab) = \varphi(a)\varphi^{i_a}(b)$ for all $b\in B$. The \emph{order} of a skew morphism, denoted by $\ord(\varphi)$, is defined as the order of the cyclic group $\langle \varphi \rangle$. Note that for each $a\in B$ there is a unique choice for $i_a$ such that $i_a\in \{1,2,\dots,\ord(\varphi)-1\}$ (unless $\varphi$ is the identity permutation). The function $\pi$ that maps each element $a\in B$ to this integer $i_a$ is called the \emph{power function} of $\varphi$, and it satisfies $\varphi(ab) = \varphi(a)\varphi^{\pi(a)}(b)$ for all $a,b\in B$. In the case when $\varphi$ is the identity permutation of $B$, we define $\pi(a)=1$ for all $a\in B$.

Skew morphisms were first introduced by Jajcay and {\v S}ir{\' a}{\v n} in~\cite{JajcaySiran}, with primary interest in their connection to the regular Cayley maps. Skew morphisms are also intriguing from a purely group-theoretical point of view, mainly due to their close relationship with group automorphisms, with which they share a number of important features. The problem of classifying all skew morphisms for given families of finite groups has gained much attention in the last two decades; see~\cite{ChenDuHeng, ConderJajcayTucker2007b, WangHuYuanZhang, ZhangDu} for example. Recently, in~\cite{BachratyConderVerret}, skew morphisms were classified for all finite simple groups, and we understand that a classification for dihedral groups is imminent; see~\cite{KovacsKwon}. On the other hand, the problem of finding all skew morphisms for finite cyclic groups remains open, despite recent positive progress, which we discuss next.

Automorphisms of a finite group $B$ are special cases of skew morphisms (with $\pi(a)=1$ for all $a\in B$), and as such can be viewed as an important family of skew morphisms. There are also other intriguing families of skew morphisms, for example, \emph{coset-preserving} (sometimes also called \emph{smooth}) skew morphisms, which are defined as skew morphisms satisfying $\pi(a)=\pi(\varphi(a))$ for all $a\in B$. Coset-preserving skew morphisms have been fully classified for all finite cyclic groups in~\cite{BachratyJajcay2017}. Another interesting family of skew morphisms that is fully understood for finite cyclic groups consists of all skew morphisms $\varphi$ such that $\varphi^2$ is an automorphism of the same group; see~\cite{HuKwonZhang}.

While there is no classification of skew morphisms of finite cyclic groups available to date, skew morphisms have been fully classified for some specific (infinite) families of finite cyclic groups. Most notably, this was done for cases where the order of a cyclic group is a prime~\cite{JajcaySiran} (in this case, all skew morphisms are automorphisms), a product of two distinct primes~\cite{KovacsNedela2011}, and any power of an odd prime~\cite{KovacsNedela2017}. Some partial progress for cyclic $2$-groups can be found in~\cite{DuHuLu}. 

Another approach for studying skew morphisms of finite cyclic groups is to find a connection between skew morphisms of a given cyclic group $B$ and skew morphisms of cyclic groups of smaller orders. Presumably the strongest finding to date made in this direction is the observation of Kov{\' a}cs and Nedela proved in~\cite{KovacsNedela2011} which states that if $\gcd(m,n) = \gcd(m,\phi(n))=\gcd(\phi(m),n)=1$, then the skew morphisms of $\ZZ_{mn}$ are exactly the direct products of skew morphisms of $\ZZ_{m}$ and $\ZZ_{n}$. There is also a useful connection between general skew morphisms and coset-preserving skew morphisms for finite cyclic groups. Namely, for each skew morphism $\varphi$ of a finite cyclic group $B$ there exists an exponent $e$ such that $\varphi^e$ is a non-trivial coset-preserving skew morphism of $B$; see~\cite{BachratyJajcay2016}.

In this paper, we combine a number of known facts about skew morphisms (which we summarise in Sections~\ref{sec:prel} and~\ref{sec:abel}) with new observations presented in Section~\ref{sec:quo}, to prove a number of theorems about skew morphisms of cyclic groups. Namely, in Section~\ref{sec:alg} we develop a new method for finding skew morphisms of cyclic groups, and implement it to obtain a census of all skew morphisms of cyclic groups of order up to $161$. (Up to the time of writing this paper, and apart from some specific orders, skew morphisms of cyclic groups were known only up to order $60$; see~\cite{ConderList}.) Further, in Section~\ref{sec:cospres} we show that all skew morphisms of $\ZZ_n$ are coset-preserving if and only if $n=2^em$ for some $e\in \{0,1,2,3,4\}$ and $m$ odd and square-free. As a consequence, we obtain a complete classification of skew morphisms for all cyclic groups of order expressible in this form, significantly expanding the list of finite cyclic groups for which such classification is available. During the review process, it was communicated to us that Kan Hu, Istv{\' a}n Kov{\' a}cs and Young Soo Kwon has recently submitted a paper devoted to similar ideas to those we investigated in Section~\ref{sec:cospres} of this paper.

\section{Preliminaries}\label{sec:prel}
In this section, we recall some definitions from group theory and provide some background from the theory of skew morphisms. All groups considered in this paper are assumed to be finite. For the cyclic group of order $n$ we use the additive notation $\ZZ_n$, and so the elements of $\ZZ_n$ may be viewed as integers in the interval $\left[ 0, n-1 \right]$. We also let $\Sym(G)$ denote the symmetric group on (the underlying set of) a group $G$. 

The \emph{core} of a subgroup $H$ in a group $G$ is the largest normal subgroup of $H$ contained in $G$. We say that $H$ is \emph{core-free} in $G$ if the core of $H$ in $G$ is trivial. A \emph{complement} for $H$ in $G$ is a subgroup $K$ of $G$ such that $G=HK$ and $H \cap K = \{ 1 \}$. The following theorem proved by Lucchini in~\cite{Lucchini} will be helpful.

\begin{theorem}[\cite{Lucchini}]
\label{thm:Lucc}
Let $C$ be a cyclic proper subgroup of a group $G$. If $C$ is core-free in $G$, then $|C|<|G:C|$.
\end{theorem}

Next, let $\varphi$ be a skew morphism of a group $B$, and identify $B$ with the subgroup of $\Sym(B)$ which acts by left multiplication. Then it can be easily checked that $B\langle \varphi \rangle$ is a subgroup of $\Sym(B)$ (see~\cite{KovacsNedela2011} for example). Moreover, $B \langle \varphi \rangle$ is a complementary factorisation and $\langle \varphi \rangle$ is core-free in $B\langle \varphi \rangle$ (see~\cite[Lemma 4.1]{ConderJajcayTucker2016}). A group $G$ containing $B$ which has a cyclic core-free complement $C$ for $B$ is called a \emph{skew product group} for a group $B$, and we say that $C$ is a \emph{skew complement} (for $B$ in $G$). The skew product group $B\langle \varphi \rangle$ (for $B$) with skew complement $\langle \varphi \rangle$ described in this paragraph is said to be \emph{induced} by $\varphi$.

Conversely, let $G$ be a skew product group for a group $B$, and let $c$ be a generator of a skew complement for $B$ in $G$. Note that every element $g\in G$ is uniquely expressible in a form $g=ac'$ with $a\in B$ and $c' \in C$. Then for every $a\in B$ there exists a unique $a' \in B$ and a unique exponent $j\in \{ 1, 2, \dots, |C|-1 \}$ such that $ca=a'c^j$, and this induces a bijection $\varphi\!: B \to B$ and a function $\pi\!: B \to \NN$, defined by $\varphi(a) = a'$ and $\pi(a) = j$. It can be easily checked that $\varphi$ is a skew morphism of $B$ with power function $\pi$. We say that $\varphi$ is \emph{induced} by the pair $(B,c)$.

Recall that if $\varphi$ is a skew morphism of $B$ with power function $\pi$, then we have $\varphi(ab) = \varphi(a)\varphi^{\pi(a)}(b)$ for all $a,b \in B$. (Hence, if $B=\ZZ_n$, then $\varphi(a+b) = \varphi(a)+\varphi^{\pi(a)}(b)$ for all $a,b \in \ZZ_n$.) Also recall that an automorphism of $B$ is a skew morphism with $\pi(a)=1$ for all $a\in B$. In what follows, it will be often convenient to distinguish between general skew morphisms and skew morphisms that are not automorphisms, and so we will refer to the latter as \emph{proper} skew morphisms. We say that $\varphi$ is \emph{trivial} if it is the identity permutation of $B$. The \emph{kernel} of $\varphi$, denoted by $\ker \varphi$, is the subset $\{a\in B \mid \pi(a) = 1 \}$ of $B$. By definition, $\varphi$ is an automorphism of $B$ if and only if $\ker \varphi = B$. In the case of proper skew morphisms $\ker \varphi$ is not equal to $B$, but it is always a subgroup of $B$; see~\cite[Lemma 4]{JajcaySiran}. 

Since $\ker\varphi$ is a subgroup of $B$ and also $\varphi(ab)=\varphi(a)\varphi(b)$ for all $a,b\in \ker\varphi$, it follows that $\varphi$ restricts to a group isomorphism from the kernel to its image. In particular, if $\ker \varphi$ is preserved by $\varphi$ set-wise, then $\varphi$ restricts to an automorphism of $\ker\varphi$. In~\cite{ConderJajcayTucker2007} this was shown to always be true for abelian groups. We also have the following.

\begin{lemma}[\cite{JajcaySiran}]
\label{lem:kernelpowerfun}
Let $\varphi$ be a skew morphism of a group $B$ with power function $\pi$. Then two elements $a,b\in B$ belong to the same right coset of the subgroup $\ker \varphi$ in $B$ if and only if $\pi(a)=\pi(b)$.
\end{lemma}

\begin{theorem}[\cite{ConderJajcayTucker2016}]
\label{thm:kernontrivial}
Every skew morphism of a non-trivial group has non-trivial kernel.
\end{theorem}

An immediate consequence of Theorem~\ref{thm:kernontrivial} is that every skew morphism of a group of prime order is an automorphism. The following facts about kernels of skew morphisms will be useful, too.

\begin{lemma}[\cite{ConderJajcayTucker2016}]
\label{lem:kernelskewprod}
Let $G$ be a skew product group for a group $B$ with skew complement $C$, and let $c$ be a generator of $C$. If~$\varphi$ is the skew morphism induced by $(B,c)$, then $\ker \varphi$ is the largest subgroup $H$ of $B$ for which $cHc^{-1}\subseteq B$. In particular, $\varphi$ is an automorphism of $B$ if and only if $B$ is normal in $G$.
\end{lemma}

\begin{proposition}[\cite{ConderJajcayTucker2016}]
\label{prop:skew induced}
Let $\varphi$ be a skew morphism of a group $B$, and let $N$ be a subgroup of $\ker\varphi$ that is normal in $B$ and preserved by $\varphi$. Then the mapping $\varphi_N^*\!: B/N \to B/N$ given by $\varphi_N^*(x) = N\varphi(x)$ is a well-defined skew morphism of $B/N$.
\end{proposition}

\begin{lemma}[\cite{ConderJajcayTucker2016}]
\label{lem:tech(a)}
Let $\varphi$ be a skew morphism of a finite abelian group $A$ with power function $\pi$. Also suppose that $N$ is any non-trivial subgroup of $\ker \varphi$ preserved by $\varphi$, let $i$ be the exponent of $N$, and let $\varphi_N^*$ be the skew morphism of $A/N$ induced by $\varphi$. If $a$ is an element of $A$ such that $Na$ lies in the kernel of $\varphi_N^*$, then $i\pi(a)\equiv i \pmod{\ord(\varphi)}$ and, in particular, if $\gcd(i,\ord(\varphi))=1$, then $a\in \ker\varphi$.  
\end{lemma}

Note that if $\varphi$ is a skew morphism of $\ZZ_n$, then $\ker\varphi$ is normal in $\ZZ_n$ and $\varphi$ restricts to an automorphism of $\ker\varphi$. Moreover, since $\ZZ_n$ is cyclic, so is (its subgroup) $\ker\varphi$. Noting that every automorphism of a cyclic group preserves all of its subgroups, we deduce that $\varphi$ preserves all subgroups of $\ker\varphi$. Hence it follows from Proposition~\ref{prop:skew induced} that $\varphi_N^*$ is a well defined skew morphism of $\ZZ_n/N$ for each subgroup $N$ of $\ker\varphi$. In the case when $N=\ker\varphi$, we will write simply $\varphi^*$ instead of $\varphi_N^*$.

We proceed with some useful facts about orders of skew morphisms.

\begin{proposition}[\cite{KovacsNedela2011}]
\label{prop:skewgenorbit}
Let $\varphi$ be a skew morphism of a group $B$, and let $T$ be an orbit of $\langle \varphi \rangle$. If $\langle T \rangle = B$, then $\ord(\varphi) = |T|$. 
\end{proposition}

\begin{theorem}[\cite{ConderJajcayTucker2016}]
\label{thm:orderofskew}
The order of a skew morphism of a non-trivial group $B$ is less than the order of $B$.
\end{theorem} 

\begin{theorem}[\cite{KovacsNedela2011}]
\label{thm:orderofskewcyclic}
If $\varphi$ is a skew morphism of a group $\ZZ_n$, then $\ord(\varphi)$ is a divisor of $n\phi(n)$. Moreover, if $\gcd(\ord(\varphi),n)=1$, then $\varphi$ is an automorphism of $\ZZ_n$.
\end{theorem}

The \emph{periodicity} of a skew morphism $\varphi$ of a group $B$, denoted by $p_{\varphi}$, is the smallest positive integer such that $\pi(a)=\pi(\varphi^{p_{\varphi}}(a))$ for all $a\in B$. Similarly, the \emph{periodicity} of $a\in B$ (with respect to $\varphi$) is the smallest positive integer $p_a$ such that $\pi(a)=\pi(\varphi^{p_a}(a))$. Note that if $\ker\varphi$ is preserved by $\varphi$ (which is always true if $B$ is abelian), then the periodicity of $\varphi$ can be defined equivalently as the order of $\varphi^*$.

Recall that $\varphi$ is coset-preserving if $\pi(a)=\pi(\varphi(a))$ for all $a\in B$. Equivalently, coset-preserving skew morphisms can be viewed as skew morphism that preserves all right cosets of $\ker\varphi$ in $B$, or as skew morphisms with the periodicity equal to $1$. The following theorem about periodicities underlines the importance of coset-preserving skew morphisms in the study of skew morphisms of abelian (and, in particular, cyclic) groups. 

\begin{theorem}[\cite{BachratyJajcay2016}]
\label{thm:powperiod}
If $\varphi$ is a skew morphism of an abelian group $A$, then $\varphi^{p_\varphi}$ is a coset-preserving skew morphism of $A$. Moreover, if $A$ is cyclic, $b$ is a generator of $A$, and $\varphi$ is non-trivial, then $p_\varphi=p_b$ and $p_\varphi < \ord(\varphi)$.
\end{theorem}

The following fact will be helpful too.

\begin{lemma}[\cite{ConderJajcayTucker2016}]
\label{lem:product}
Let $\varphi$ be any skew morphism of a finite group $G$, and let $H$ be any finite group. Then $\varphi$ can be extended to a skew morphism $\theta$ of $G\times H$, such that $\theta \restriction_G = \varphi$ and $\ker\theta = \ker\varphi \times H$.
\end{lemma}

We conclude this section with two well-known facts about skew morphisms that are easy exercises, but we include their proofs for completeness.

\begin{lemma}\label{lem:fix}
Let $\varphi$ be a skew morphism of an abelian group $A$. If $\varphi(a)=a$ for some $a\in A$, then $a\in \ker \varphi$.
\end{lemma}
\begin{proof}
Let $a'$ be any element of $A$. Since $a$ is fixed by $\varphi$, we have $\varphi(aa')=a\varphi^{\pi(a)}(a')$. On the other hand, we have $\varphi(aa')=\varphi(a'a)=\varphi(a')\varphi^{\pi(a')}(a)=\varphi(a')a$, and hence $\varphi(a')=\varphi^{\pi(a)}(a')$ for all $a'\in A$. It follows that $\pi(a)=1$, and therefore $a\in \ker \varphi$.
\end{proof}

\begin{lemma}
\label{lem:skewfromorbit}
Let $\varphi$ be a skew morphism of a group $B$ with power function $\pi$, and let $a\in B$. Then:
\begin{align*}
\varphi(a^i) = \varphi(a) \varphi^{\pi(a)}(a) \varphi^{\pi(a^2)}(a)  \dots  \varphi^{\pi(a^{i-1})}(a) \hbox{ for all } i\in \NN .
\end{align*}
\end{lemma}
\begin{proof}
The assertion is trivially true for $i = 1$. Next, if it holds for some positive integer $j$, then $\varphi(a^{j+1}) = \varphi(a^j a) = \varphi(a^j)\varphi^{\pi(a^j)}(a) = \varphi(a) \varphi^{\pi(a)}(a) \varphi^{\pi(a^2)}(a) \dots  \varphi^{\pi(a^{j-1})}(a) \varphi^{\pi(a^j)}(a)$, and so it is also true for $j+1$. Hence the proof follows by induction.
\end{proof}

\section{Skew morphisms of abelian groups}\label{sec:abel}

Several useful facts about skew morphisms of abelian groups (which we also apply in this paper) were proved in~\cite{ConderJajcayTucker2016}. During the preparation of this paper, however, we noticed that three assertions in~\cite{ConderJajcayTucker2016} do not hold. In this section, we list the incorrect findings and provide a counterexample for each of them. We also propose and prove weaker versions of the original statements. Finally, we discuss all proofs in~\cite{ConderJajcayTucker2016} that use at least one of the flawed statements, and show that in all cases it is sufficient to replace the flawed statements by our modifications. For the rest of the section, we let $\varphi$ be a skew morphism of an abelian group $A$ with power function $\pi$. Also, to distinguish between references to this paper and references to~\cite{ConderJajcayTucker2016}, we put an asterisk after each numbered reference in the latter case. 

The first flawed assertion in ~\cite{ConderJajcayTucker2016} is part (b) of Lemma~5.3$^*$. It states that if $N$ is a non-trivial subgroup of $\ker \varphi$ preserved by $\varphi$, and $\varphi$ is not an automorphism of $A$, then $\ord(\varphi)$ has a non-trivial divisor in common with the exponent of $N$. To show that this is not true, note that according to \cite{ConderList} the cyclic group of order $12$ admits a proper skew morphism $\psi$ of order $3$ with kernel (which is preserved by $\psi$) of order $6$. Since $\ker \psi$ is cyclic and preserved by $\psi$, so is its unique subgroup of order $2$. But the exponent of this subgroup, which is $2$, does not have a non-trivial divisor in common with $\ord(\psi)$. We propose the following modification:

\begin{lemma}
\label{lem:tech(b)}
If $\varphi$ is a proper skew morphism of an abelian group $A$, then $\ord(\varphi)$ has a non-trivial divisor in common with the exponent of $\ker\varphi$.
\end{lemma}
\begin{proof}
Let $K=\ker \varphi$, let $e$ be the exponent of $K$, and let $L/K$ be the kernel of the skew morphism $\varphi^*$ of $A/K$ induced by $\varphi$. Since $\varphi$ is proper we know that $A/K$ is a non-trivial group, and by Theorem~\ref{thm:kernontrivial} it follows that $L/K$ is non-trivial. In particular, there exists an element $a$ of $A$ such that $a\notin K$ and $a\in L$. It follows that $\pi(a)\not\equiv 1 \pmod{\ord(\varphi)}$, by Lemma~\ref{lem:tech(a)} we have $e\pi(a)\equiv e \pmod{\ord(\varphi)}$, and the rest follows.
\end{proof}

Part (b) of Lemma~5.3$^*$ is used in the proofs of six theorems presented in~\cite{ConderJajcayTucker2016}. Proofs of Theorem~5.4$^*$, Theorem~5.10$^*$, Theorem~6.2$^*$ and Theorem~6.4$^*$ are all easily fixable, since in each case Lemma~5.3(b)$^*$ is applied for $N=K$, and hence it is sufficient to replace it with Lemma~\ref{lem:tech(b)}. The proof of Theorem~6.1$^*$ can also be corrected. Here Lemma~5.3(b)$^*$ is used to show that if $\varphi$ is proper and $A \cong \ZZ_n$, then $\gcd(\ord(\varphi),n) \neq 1$. Since $A$ is cyclic, the exponent of $\ker \varphi$ is equal to the order of $\ker\varphi$ (which necessarily divides $n$), so this is an easy consequence of Lemma~\ref{lem:tech(b)}. Finally, since Theorem~7.3$^*$ was already proved previously in~\cite{KovacsNedela2011}, there is no need to fix its alternative proof presented in~\cite{ConderJajcayTucker2016}; although we believe that it is possible.

Another flawed assertion in ~\cite{ConderJajcayTucker2016} is Theorem~5.6$^*$. It states that if $L/(\ker \varphi)$ is the kernel of the skew morphism $\varphi^*$ of $A/(\ker \varphi)$ induced by $\varphi$, and $p$ is a prime that divides $|L|$ but not $|\ker \varphi|$, then $p<q$ for every prime divisor $q$ of $|\ker \varphi|$. Again, we provide a counterexample that shows that this is not true. According to~\cite{ConderList} the cyclic group of order $42$ admits a proper skew morphism $\rho$ of order $7$ with kernel of order $14$. Since $\ZZ_{42}/(\ker \rho)$ is isomorphic to $\ZZ_3$ and $\ZZ_3$ does not admit any proper skew morphism, it follows that the kernel $L/(\ker \rho)$ (of the skew morphism of $\ZZ_{42}/(\ker \rho)$ induced by $\rho$) is equal to $\ZZ_{42}/(\ker \rho)$. Therefore, $|L|=|\ZZ_{42}|$ and, in particular, $3$ divides $|L|$. But $3$ is greater than $2$, and $2$ is a prime divisor of $|\ker \rho|$. We propose the following modification:

\begin{theorem}
\label{thm:corr}
Let $\varphi$ be a skew morphism of the finite abelian group $A$, let $K=\ker\varphi$, and let $q$ be any prime divisor of $|K|$. Also let $N$ be a subgroup of $K$ consisting of the identity and all elements of order $q$, and let $L/N$ be the kernel of the skew morphism $\varphi_N^*$ of $A/N$ induced by $\varphi$. If $p$ is a prime that divides $|L|$ but not $|K|$, then $p<q$.
\end{theorem}
\begin{proof}
Suppose that such a prime $p$ exists. Since $K$ is abelian, we know that $N$ is a subgroup of $K$ of exponent $q$ that is invariant under $\varphi$. Next, let $a$ be any element of order $p$ in $L$, let $m=\ord(\varphi)$, and let $\pi$ be the power function of $\varphi$. Since $L/N$ is the kernel of $\varphi_N^*$, we know by Lemma~\ref{lem:tech(a)} that $q(\pi(a)-1) \equiv 0 \pmod{m}$. If $q$ is relatively prime to $m$, then $\pi(a)\equiv 1 \pmod{m}$ and so $a\in K$, which is impossible since $K$ has no element of order $p$. Thus $q$ divides $m$ and $\pi(a)-1 \equiv 0 \pmod{m/q}$. In particular, $\pi(a) = 1 + i(m/q)$ where $1\leq i \leq q-1$, so there are at most $q-1$ possibilities for $\pi(a)$.

The same holds for every non-trivial power of $a$. So now if $p>q$, then by the pigeonhole principle two different powers of $a$ will have the same value under $\pi$, in which case they lie in the same coset of $N$. But that cannot happen since $K \cap \langle a \rangle$ is trivial. Thus $p<q$.
\end{proof}

The only application of the flawed original version of Theorem~5.6$^*$ is Corollary~5.7$^*$. This final problematic assertion in~\cite{ConderJajcayTucker2016} states that every prime divisor of $|\ker \varphi|$ is greater than every prime that divides $|A|$ but not $|\ker \varphi|$. To see that this is not true, take the skew morphism $\rho$ of $\ZZ_{42}$ with kernel of order $14$ discussed earlier. Since $2$ divides $|\ker \rho|$ and $3$ divides $|\ZZ_{42}|$ but not $|\ker \rho|$, the statement is clearly not true. The following, however, still holds.

\begin{corollary}
\label{cor:order kernel}
Let $A$ be a non-trivial finite abelian group, and let $p$ be the largest prime divisor of $|A|$. Then the order of the kernel of every skew morphism of $A$ is divisible by $p$ when $p$ is odd, or by $4$ when $p=2$. 
\end{corollary}
\begin{proof}
Let $\varphi$ be any skew morphism of $A$, let $K=\ker \varphi$, and suppose to the contrary that $p$ does not divide $|K|$. Also let $q$ be any prime divisor of $|K|$, let $N$ be a subgroup of $K$ consisting of the identity and all elements of order $q$, and let $L/N$ be the kernel of the skew morphism $\varphi_N^*$ of $A/N$ induced by $\varphi$. If $p$ divides $|L/N|$, then by Theorem~\ref{thm:corr} we know that $p$ is smaller than $q$, so this cannot happen. It follows that $p$ does not divide $|L/N|$, so we can repeat the same argument for the skew morphism $\varphi_N^*$ of $A/N$ with kernel $L/N$. (Note that $p$ divides $|A/N|$.) Since $A$ is finite, this will eventually terminate for some groups $A'$, $N'$ and $L'$ with $|L'/N'|=1$. Then since the kernel $L'/N'$ is trivial, it follows by Theorem~\ref{thm:kernontrivial} that $A'/N'$ is a trivial group, and hence $|A'| = |N'|$. But this is impossible, since $p$ divides $|A'|$ but not $|N'|$. The second part for $p=2$ follows from the original proof of~\cite[Corollary 5.7]{ConderJajcayTucker2016}. 
\end{proof}

Both applications of Corollary~5.7$^*$, namely the proofs of Theorem~6.2$^*$ and Theorem~9.1$^*$, only use the fact that the order of $\ker \varphi$ is divisible by the largest prime divisor of $|A|$. Since this follows by Corollary~\ref{cor:order kernel}, both theorems still hold, and their proofs can be corrected by minor changes in their wording.

\section{Quotients and their properties}\label{sec:quo}

In this section, we will show that if $BC$ is a complementary product of two cyclic groups with $C$ core-free in $BC$, then not only $C$ corresponds to some skew morphism of $B$ (of order $|C|$), but also $B$ corresponds to some skew morphism of $C$ (of order smaller than $|B|$). Let $\varphi$ be a skew morphism of a cyclic group $B$, let $G=B\langle \varphi \rangle$, and let $C=\langle \varphi \rangle$. Also let $b$ be a generator of $B$. To distinguish between $\varphi$ as a permutation of $B$ and $\varphi$ as a generator of the cyclic group $C$, we use $c$ in the latter case. Since $C$ is core-free in $G$, by Theorem~\ref{thm:Lucc} we have 
\[ |G:B| = |C| < |G:C| = |B|\, , \]
so (again by Theorem~\ref{thm:Lucc}) we find that $B$ has a non-trivial core in $G$. Let $K$ denote the core of $B$ in $G$, and for every $X\leq G$ let $\overline{X}$ denote $XK/K$ ($\cong X/(X \cap K)$). Next, let $H$ be a subgroup of $B$ such that $cHc^{-1} \subseteq B$. Then, since $B$ is cyclic, $H$ is the unique subgroup of $B$ of order $|H|$, and so $cHc^{-1}=H$. It follows that $H$ is normal in $G$, and so it is contained in $K$. Moreover, since $K$ is the core of $B$ in $G$, we have $cKc^{-1}=K\subseteq B$, and hence by Lemma~\ref{lem:kernelskewprod} we find that $K = \ker \varphi$.

Now we look closely at the product $\overline{G}=\overline{B}\hskip 2pt \overline{C}$. First, noting that $K \cap C = \{ 1 \}$ we have $C \cong \overline{C}$ (and so $\overline{C}$ is cyclic, and hence abelian) and $\overline{B}\cap \overline{C} = \{ 1 \}$. Since $B$ is cyclic, so is its quotient $\overline{B}$, and by the definition of $K$ we deduce that $\overline{B}$ is core-free in $\overline{G}$. Now it is straightforward to check that the bijection that maps every element $\overline{d}\in \overline{C}$ to the unique element $\overline{d'} \in \overline{C}$ such that $\overline{d}\hskip 2pt \overline{b} = \overline{b}^j\overline{d'}$ defines a skew morphism of $\overline{C}$, with power function $\overline{\pi}$ given by $\overline{\pi}(\overline{d}) = j$. (We choose $\overline{b}$ to be the image of $b\in B$ under the natural homomorphism from $B$ to $\overline{B}$; since $b$ is a generator of $B$, it follows that $\overline{b}$ is a generator of $\overline{B}$.)

A skew morphism of $C$ ($\cong \overline{C}$) constructed in the way described in the previous paragraph is called the \emph{quotient} of $\varphi$ (with respect to $b$), and will be denoted by $\overline{\varphi}$. Since a cyclic group $\overline{B}$ can be generated by different elements, $\varphi$ can have more than one quotient. (In fact, by~\cite[Remark 5.1]{BachratyConderVerret} this is always true unless B has a unique generator.) In the case of additive notation $B=\ZZ_n$, and unless otherwise specified, by the quotient of $\varphi$ we understand the quotient with respect to $1$.

The above construction (proposed in the author's PhD thesis \cite{phd}) was also introduced independently in~\cite{FengHu}, where it was noted that if $\varphi$ is a skew morphism of $\ZZ_n$ and $\overline{\varphi}$ is a quotient of $\varphi$, then $(\varphi,\overline{\varphi})$ is an $(\ord(\varphi),n)$-reciprocal pair of skew morphisms. A pair $(\varphi,\rho)$ of skew morphisms of $\ZZ_n$ and $\ZZ_m$ with power functions $\pi$ and $\tau$ is called \emph{$(m,n)$-reciprocal} if $\ord(\varphi)$ divides $m$, $\ord(\rho)$ divides $n$, and the congruences $\pi(i)\equiv \rho^i(1) \pmod{\ord(\varphi)}$ and $\tau(j)\equiv\varphi^j(1) \pmod{\ord(\rho)}$ hold for each $i\in \ZZ_n$ and $j\in \ZZ_m$. We note that while every pair $(\varphi,\overline{\varphi})$ gives an $(\ord(\varphi),n)$-reciprocal pair of skew morphisms, not every reciprocal pair arises in this way. For example, there exist $(m,n)$-reciprocal pairs of skew morphisms with $m=n$, but by Theorem~\ref{thm:orderofskew} we know that $\ord(\varphi)$ is always strictly smaller than $n$.

The following observation is an easy consequence of the fact that a skew morphism (of a cyclic group) and its quotient always give a reciprocal pair of skew morphisms. 

\begin{lemma}
\label{lem:quotientmain}
Let $\varphi$ be a skew morphism of $\ZZ_n$ with power function $\pi$, and let $\overline{\varphi}$ be the quotient of $\varphi$ with power function $\overline{\pi}$. Then for every $i\in \NN$:
\begin{enumerate}[label={\rm (\alph*)},ref=\ref{lem:quotientmain}(\alph*)]
\item\label{lem:quotientmain:a} $\pi(i) = \overline{\varphi}^{\, i}(1)$, so in particular $\ord(\overline{\varphi})=n/|\ker\varphi|$; and
\item\label{lem:quotientmain:b} $\varphi^i(1) \equiv \overline{\pi}(i) \pmod{n/|\ker\varphi|}$.
\end{enumerate}
\end{lemma}
\begin{proof}
First, since $(\varphi,\overline{\varphi})$ is an $(\ord(\varphi),n)$-reciprocal pair of skew morphisms, we have $\pi(i)\equiv \overline{\varphi}^{\, i}(1) \pmod{\ord(\varphi)}$. Hence, since both $\pi$ and $\overline{\varphi}$ are mappings into $\ZZ_{\ord(\varphi)}$, it follows that $\pi(i) = \overline{\varphi}^{\, i}(1)$. The second part of (a) follows from the fact that $n/|\ker\varphi|$ is the smallest non-zero integer in $\ker\varphi$. Finally, (b) follows easily as $\overline{\pi}(i)\equiv\varphi^i(1) \pmod{\ord(\overline{\varphi})}$ and $\ord(\overline{\varphi})=n/|\ker\varphi|$. 
\end{proof}

Next we provide a lemma which shows that quotients can be used to check whether a skew morphism of a cyclic group is an automorphism or a coset-preserving skew morphism. 

\begin{lemma}
\label{lem:quotientcp}
A skew morphism $\varphi$ of $\ZZ_n$ is coset-preserving if and only if the quotient $\overline{\varphi}$ of $\varphi$ is an automorphism. Moreover, $\varphi$ is proper if and only if $\overline{\varphi}$ is non-trivial. 
\end{lemma}
\begin{proof}
Let $\varphi$ be a coset-preserving skew morphism of $\ZZ_n$. This is equivalent with $\varphi(1)\equiv 1 \pmod{n/|\ker\varphi|}$, which by Lemma~\ref{lem:quotientmain:b} happens if and only if $\overline{\pi}(1)=1$. This proves the first part. The second part follows easily by Lemma~\ref{lem:quotientmain:a} as $\varphi \in \Aut(\ZZ_n)$ if and only if $\pi(1)=1$, and $\overline{\varphi}$ is trivial if and only if $\overline{\varphi}(1)=1$.   
\end{proof}

We note that a part of the Lemma~\ref{lem:quotientcp} was proved in~\cite{HuNedela}, where it was shown that if $\overline{\varphi}$ is an automorphism, then $\varphi$ is coset-preserving, but not the other way around. Also note that since the only skew morphism of $\ZZ_2$ is the identity mapping, it follows immediately from Lemma~\ref{lem:quotientcp} that if a skew morphism of a cyclic group has order $2$, then it must be an automorphism. (This is also an easy consequence of Lemma~\ref{lem:kernelskewprod}.) Another consequence of Lemma~\ref{lem:quotientcp} is the fact that a proper skew morphism of $\ZZ_n$ is coset-preserving if and only if the quotient of its quotient is the identity mapping. Somewhat interestingly, Lemma~\ref{lem:quotientcp} also implies that by taking quotients every skew morphism of a cyclic group can be reduced to a non-trivial automorphism of the same group.

\section{Using quotients to generate skew morphisms}\label{sec:alg}

In this section, we describe an algorithm for finding recursively skew morphisms of cyclic groups based on various observations about the quotients of skew morphisms.

\subsection{Skew morphisms with a given quotient}

First, we explain how to find all skew morphisms of a cyclic group with a given quotient. The following observation about quotients of skew morphisms of cyclic groups will be useful.

\begin{proposition}
\label{prop:quotientinv}
Let $\varphi$ be a skew morphism of $\ZZ_n$ with power function $\pi$, and let $\overline{\varphi}$ be the quotient of $\varphi$ with power function $\overline{\pi}$. Then:
\begin{enumerate}[label={\rm (\alph*)},ref=\ref{prop:quotientinv}(\alph*)]
\item\label{prop:quotientinv:a} the periodicity $p_{\varphi}$ of $\varphi$ is the smallest generator of $\ker \overline{\varphi}$; 
\item\label{prop:quotientinv:b} $\ord(\varphi)=|\ker\overline{\varphi}| p_{\varphi}$; and
\item\label{prop:quotientinv:c} the orbit $T$ of $\langle \varphi \rangle$ that contains $1$ is expressible in the form
\begin{equation}
\label{eq:quotientinv}
T = (x_1,\dots,x_{p_{\varphi}}, \psi(x_1),\dots, \psi(x_{p_{\varphi}}), \dots, \psi^{\ord(\psi)-1}(x_1), \dots, \psi^{\ord(\psi)-1}(x_{p_{\varphi}})),
\end{equation}
for some coset-preserving skew morphism $\psi$ of $\ZZ_n$ such that $\ord(\psi) = \ord(\varphi) / p_{\varphi}$ and $\psi(1) \equiv 1 \pmod{n/|\ker\varphi|}$, and with $x_1 = 1$ and $x_i \equiv \overline{\pi}(i-1) \pmod{n/|\ker\varphi|}$ for each $i \in \{2,\dots,p_{\varphi}\}$.
\end{enumerate}
\end{proposition}
\begin{proof}
First, by Theorem~\ref{thm:powperiod} we have $p_{\varphi} = p_1$. Then, since the values taken by $\pi$ at any two elements of $\ZZ_n$ are equal if and only if they belong to the same right coset of $\ker\varphi$ in $\ZZ_n$, it follows 
that $p_1$ is the smallest positive integer such that $1 \equiv \varphi^{p_1}(1) \pmod{n/|\ker\varphi|}$, which by Lemma~\ref{lem:quotientmain:b} is equivalent with $1 \equiv \overline{\pi}(p_1) \pmod{n/|\ker\varphi|}$. Noting that $n/|\ker\varphi|=\ord(\overline{\varphi})$, we deduce that $p_1$ is the smallest positive integer such that $\overline{\pi}(p_1)=1$, and (a) follows. Moreover, since $p_{\varphi}$ is the smallest non-trivial element of $\ker\overline{\varphi}$, which is a subgroup of $\ZZ_{\ord(\varphi)}$, it follows that $\ord(\varphi)=|\ker\overline{\varphi}|p_{\varphi}$, which proves (b).  

To prove the final assertion, let $T$ denote the orbit of $\langle \varphi \rangle$ that contains $1$, and let $\psi = \varphi^{p_1}$. By Theorem~\ref{thm:powperiod} we know that $\psi$ is a coset-preserving skew morphism of $\ZZ_n$, and by the definition of the periodicity we have $\psi(1) \equiv 1 \pmod{n/|\ker\varphi|}$. By Proposition~\ref{prop:skewgenorbit} we know that the size of $T$ is equal to $\ord(\varphi)$. Moreover, $p_1$ divides $|T|$ (see \cite[Lemma 3.1]{BachratyJajcay2017}), and hence $\ord(\psi) = \ord(\varphi^{p_1}) = \ord(\varphi)/p_1$, and the effect of $\psi$ on $T$ induces $p_1$ cycles, each of length $\ord(\varphi)/p_1$. Finally, by Lemma~\ref{lem:quotientmain:b} we have $x_i = \varphi^{i-1}(1) \equiv \overline{\pi}(i-1) \pmod{n/|\ker\varphi|}$ and the rest follows.
\end{proof}

Let $\rho$ be a skew morphism of a cyclic group $\ZZ_m$. We will provide a detail explanation of the method for finding all skew morphisms $\varphi$ of $\ZZ_n$ with quotient $\rho$.\footnote{It is important to emphasise here that this method finds all skew morphisms with a particular quotient only for a given cyclic group. If we do not restrict ourselves to a specific group, then in some cases we can find infinitely many skew morphisms with a given quotient. For example, using the classification of skew morphisms for cyclic $p$-groups of odd order (presented in \cite{KovacsNedela2017}) it can be shown that each cyclic $3$-group of order at least $9$ admits a skew morphism whose quotient is the skew morphism $(1,3,5)$ of $\ZZ_6$.}     

First, we find the smallest positive integer $j$ such that $j \in \ker \rho$. Then by Proposition~\ref{prop:quotientinv:a} we have $p_1 = p_\varphi = j$. Next, since $\rho$ is a quotient of $\varphi$, it follows by Lemma~\ref{lem:quotientmain:a} that $\ord(\rho) = n / |\ker\varphi|$, and hence $\ker \varphi$ is the unique subgroup of $\ZZ_n$ of order $n/\ord(\rho)$. As a next step we find all coset-preserving skew morphisms $\psi$ of $\ZZ_n$ satisfying $\ord(\psi) = m/p_{\varphi}$ and $\psi(1) \equiv 1 \pmod{n/|\ker\varphi|}$ (note here that $m = |\ZZ_m| = \ord( \varphi)$). This allows us to identify all possible candidates for the orbit $T$ of $\langle \varphi \rangle$ that contains $1$, using~\eqref{eq:quotientinv}. (For each choice of $\psi$, we have at most $|\ker \varphi|^{p_1 - 1} = (n/\ord(\rho))^{p_1-1}$ candidates; the number of candidates could be smaller since sometimes~\eqref{eq:quotientinv} does not define a cycle on $B$.)

Next, suppose that $\varphi$ is a skew morphism, and let $\varphi_1$ be the cyclic permutation of $T$ induced by $\varphi$. Then by Lemma~\ref{lem:quotientmain:a} we have $\pi(i) = \rho^i(1)$, and hence by Lemma~\ref{lem:skewfromorbit} we find that
\begin{equation}
\label{eq:quodef}
\varphi(i) = \varphi_1(1) + \varphi_1^{\, \rho(1)}(1) + \varphi_1^{\, \rho^2(1)}(1) + \dots +  \varphi_1^{\, \rho^{i-1}(1)}(1)\ \hbox{ for all } i\in \NN.
\end{equation}

As a final step, for each candidate for $\varphi_1$ we use~\eqref{eq:quodef} to define a function $\varphi\!: \ZZ_n \to \ZZ_n$, and then check whether $\varphi$ is a skew morphism of $\ZZ_n$, and $T$ an orbit of $\langle \varphi \rangle$. It can be easily verified that if this is true, then $\rho$ is the quotient of $\varphi$.  

\begin{remark}
Note that to use~\eqref{eq:quodef}, we do not need to know the complete orbit $T$, but only the elements $\varphi_1^{\, \rho^i(1)}(1)$ for each $i\in \{0, 1, \dots, \ord(\rho)\}$. In fact, since for every positive integer $j$ we have $\varphi_1^{\, j+p_1}(1) = \psi(\varphi_1^{\, j}(1))$, only elements of the form $\varphi_1^{\, e}(1)$ with $e \equiv \rho^i(1) \pmod{p_1}$ are needed to define $\varphi$. In some cases, this significantly reduces the number of possible candidates for $\varphi_1$. 
\end{remark}

\subsection{Algorithm for finding all skew morphisms of a cyclic group}\label{subsec:alg}

We are ready to describe our algorithm for finding all skew morphisms of cyclic groups up to any order. This algorithm is recursive in the sense that it takes the sets of all skew morphisms of the groups $\ZZ_m$ for $m \in \{2, 3, \dots, n-1\}$ as input and outputs all skew morphisms of $\ZZ_n$. Since the only skew morphism of $\ZZ_2$ is the identity permutation, the algorithm can be easily initialised.

As the first step, we use the method presented in~\cite{BachratyJajcay2017} to find all coset-preserving skew morphisms of $\ZZ_n$. Further details on this method are available in Section~\ref{subsec:enum}. Next, let $\varphi$ be a skew morphism of $\ZZ_n$ that is not coset-preserving, and let $\overline{\varphi}$ be the quotient of $\varphi$. Then by Lemma~\ref{lem:quotientcp} we know that $\overline{\varphi}$ is a proper skew morphism of $\ZZ_{\ord(\varphi)}$. Moreover, by Theorem~\ref{thm:orderofskew} and Theorem~\ref{thm:orderofskewcyclic} we find that $\ord(\varphi)<n$, and that $\ord(\varphi)$ divides $n\phi(n)$, and $\gcd(\ord(\varphi),n) \neq 1$. Since we know all skew morphisms of $\ZZ_m$ for each $m<n$, and we also know all coset-preserving skew morphisms of $\ZZ_n$, it follows that we can simply apply the method explained in the previous subsection to find all skew morphisms of $\ZZ_n$ that are not coset-preserving. Together with the coset-preserving skew morphisms of $\ZZ_n$ (which include all automorphisms and were found earlier), this gives all skew morphisms of $\ZZ_n$.

A \textsc{Magma}~\cite{Magma} implementation of the described algorithm succeeded in finding all skew morphisms of cyclic groups of order up to $161$.  (The file listing all of these skew morphisms is available at~\cite{List}.)  This significantly improves the previous largest complete list~\cite{ConderList} which goes up to the order $60$. In Table~\ref{tab:cyclic} we summarise the information obtained about skew morphisms of cyclic groups $\ZZ_n$ for $n \leq 161$. We include a group in the table if and only if it admits a proper skew morphism. Moreover, if a listed group admits a proper skew morphism that is not coset-preserving, then the order of the group is preceded by the asterisk character ($^*$). All included groups are listed by their orders, and for each of them we provide the total number of skew morphisms (written as the sum of the numbers of proper skew morphisms and automorphisms), and the number of conjugacy classes of proper skew morphisms in $\Aut(\ZZ_n)$. Note that the automorphism group of a cyclic group $\ZZ_n$ is always abelian, and hence the conjugation action of $\Aut(\ZZ_n)$ on itself is trivial. It follows that the number of conjugacy classes of $\Aut(\ZZ_n)$ is equal to $|\Aut(\ZZ_n)|$. For this reason, we list the number of conjugacy classes only for proper skew morphisms. We also note that the numbers of skew morphisms in Table~\ref{tab:cyclic} for $n \leq 60$ coincide with the numbers of skew morphisms in~\cite{ConderList}.

\begin{table}
\centering
\renewcommand{\arraystretch}{0.8}
\setlength\tabcolsep{2.5pt}
\begin{tabular}{@{}rccrclccrccccrccrclccrccccrccrclccrccc@{}} \toprule
$n$ &&& \multicolumn{3}{c}{$\Skew$} && \multicolumn{3}{c}{Classes} &&&& $n$ &&& \multicolumn{3}{c}{$\Skew$} && \multicolumn{3}{c}{Classes} &&&& $n$ &&& \multicolumn{3}{c}{$\Skew$} && \multicolumn{3}{c}{Classes}\\  \cmidrule{1-10}  \cmidrule{14-23} \cmidrule{27-36}
$6$ &&& $\mathbf{2}$ & $+$ & $2$ && $\ \ $ & $\mathbf{1}$  		&&&&&   $58$ &&& $\mathbf{28}$ & $+$ & $28$ && $\ \ $ & $\mathbf{1}$	&&&&&  $114$ &&& $\mathbf{148}$ & $+$ & $36$ && $\ \ $ & $\mathbf{7}$   \\
$8$ &&& $\mathbf{2}$ & $+$ & $4$ &&& $\mathbf{1}$  			&&&&&  $60$ &&& $\mathbf{80}$ & $+$ & $16$ &&& $\mathbf{17}$   		&&&&&  $116$ &&& $\mathbf{112}$ & $+$ & $56$ &&& $\mathbf{3}$    \\
$^*9$ &&& $\mathbf{4}$ & $+$ & $6$ &&& $\mathbf{2}$  			&&&&&   $62$ &&& $\mathbf{30}$ & $+$ & $30$ &&& $\mathbf{1}$   		&&&&&  $^*117$ &&& $\mathbf{88}$ & $+$ & $72$ &&& $\mathbf{11}$    \\
$10$ &&& $\mathbf{4}$ & $+$ & $4$ &&& $\mathbf{1}$ 			&&&&&   $^*63$ &&& $\mathbf{44}$ & $+$ & $36$ &&& $\mathbf{7}$   		&&&&& $118$ &&& $\mathbf{58}$ & $+$ & $58$ &&& $\mathbf{1}$      \\
$12$ &&& $\mathbf{4}$ & $+$ & $4$ &&& $\mathbf{2}$ 			&&&&&  $^*64$ &&& $\mathbf{268}$ & $+$ & $32$ &&& $\mathbf{42}$   		&&&&& $120$ &&& $\mathbf{208}$ & $+$ & $32$ &&& $\mathbf{43}$       \\ \addlinespace[.2em]
$14$ &&& $\mathbf{6}$ & $+$ & $6$ &&& $\mathbf{1}$  			&&&&&  $66$ &&& $\mathbf{60}$ & $+$ & $20$ &&& $\mathbf{13}$  		&&&&& $^*121$ &&& $\mathbf{900}$ & $+$ & $110$ &&& $\mathbf{90}$       \\ 
$16$ &&& $\mathbf{12}$ & $+$ & $8$ &&& $\mathbf{4}$  			&&&&&   $68$ &&& $\mathbf{64}$ & $+$ & $32$ &&& $\mathbf{3}$   		&&&&& $122$ &&& $\mathbf{60}$ & $+$ & $60$ &&& $\mathbf{1}$       \\
$^*18$ &&& $\mathbf{24}$ & $+$ & $6$ &&& $\mathbf{6}$  			&&&&&  $70$ &&& $\mathbf{72}$ & $+$ & $24$ &&& $\mathbf{11}$			&&&&& $124$ &&& $\mathbf{60}$ & $+$ & $60$ &&& $\mathbf{2}$        \\
$20$ &&& $\mathbf{16}$ & $+$ & $8$ &&& $\mathbf{3}$  			&&&&&  $^*72$ &&& $\mathbf{156}$ & $+$ & $24$ &&& $\mathbf{36}$  		&&&&& $^*125$ &&& $\mathbf{1568}$ & $+$ & $100$ &&& $\mathbf{152}$        \\ 
$21$ &&& $\mathbf{12}$ & $+$ & $12$ &&& $\mathbf{1}$  			&&&&&  $74$ &&& $\mathbf{36}$ & $+$ & $36$ &&& $\mathbf{1}$   		&&&&& $^*126$ &&& $\mathbf{348}$ & $+$ & $36$ &&& $\mathbf{34}$       \\ \addlinespace[.2em]
$22$ &&& $\mathbf{10}$ & $+$ & $10$ &&& $\mathbf{1}$  			&&&&&  $^*75$ &&& $\mathbf{96}$ & $+$ & $40$ &&& $\mathbf{24}$    		&&&&& $^*128$ &&& $\mathbf{1132}$ & $+$ & $64$ &&& $\mathbf{114}$       \\
$24$ &&& $\mathbf{16}$ & $+$ & $8$ &&& $\mathbf{7}$  			&&&&& $76$ &&& $\mathbf{36}$ & $+$ & $36$ &&& $\mathbf{2}$	  		&&&&&  $129$ &&& $\mathbf{84}$ & $+$ & $84$ &&& $\mathbf{1}$      \\ 
$^*25$ &&& $\mathbf{48}$ & $+$ & $20$ &&& $\mathbf{12}$  		&&&&& $78$ &&& $\mathbf{104}$ & $+$ & $24$ &&& $\mathbf{9}$  		&&&&& $130$ &&& $\mathbf{144}$ & $+$ & $48$ &&& $\mathbf{17}$       \\
$26$ &&& $\mathbf{12}$ & $+$ & $12$ &&& $\mathbf{1}$  			&&&&& $80$ &&& $\mathbf{152}$ & $+$ & $32$ &&& $\mathbf{26}$   		&&&&& $132$ &&& $\mathbf{120}$ & $+$ & $40$ &&& $\mathbf{26}$       \\
$^*27$ &&& $\mathbf{64}$ & $+$ & $18$ &&& $\mathbf{20}$  		&&&&& $^*81$ &&& $\mathbf{676}$ & $+$ & $54$ &&& $\mathbf{110}$   		&&&&&  $134$ &&& $\mathbf{66}$ & $+$ & $66$ &&& $\mathbf{1}$      \\ \addlinespace[.2em]
$28$ &&& $\mathbf{12}$ & $+$ & $12$ &&& $\mathbf{2}$  			&&&&& $82$ &&& $\mathbf{40}$ & $+$ & $40$ &&& $\mathbf{1}$   		&&&&& $^*135$ &&& $\mathbf{256}$ & $+$ & $72$ &&& $\mathbf{80}$        \\
$30$ &&& $\mathbf{24}$ & $+$ & $8$ &&& $\mathbf{7}$  			&&&&& $84$ &&& $\mathbf{104}$ & $+$ & $24$ &&& $\mathbf{14}$  		&&&&& $136$ &&& $\mathbf{228}$ & $+$ & $64$ &&& $\mathbf{10}$       \\
$^*32$ &&& $\mathbf{60}$ & $+$ & $16$ &&& $\mathbf{14}$  		&&&&& $86$ &&& $\mathbf{42}$ & $+$ & $42$ &&& $\mathbf{1}$  			&&&&&   $138$ &&& $\mathbf{132}$ & $+$ & $44$ &&& $\mathbf{25}$     \\ 
$34$ &&& $\mathbf{16}$ & $+$ & $16$ &&& $\mathbf{1}$  			&&&&&  $88$ &&& $\mathbf{80}$ & $+$ & $40$ &&& $\mathbf{15}$   		&&&&&  $140$ &&& $\mathbf{240}$ & $+$ & $48$ &&& $\mathbf{29}$      \\
$^*36$ &&& $\mathbf{48}$ & $+$ & $12$ &&& $\mathbf{12}$  		&&&&& $^*90$ &&& $\mathbf{216}$ & $+$ & $24$ &&& $\mathbf{36}$   		&&&&&  $142$ &&& $\mathbf{70}$ & $+$ & $70$ &&& $\mathbf{1}$       \\  \addlinespace[.2em]
$38$ &&& $\mathbf{18}$ & $+$ & $18$ &&& $\mathbf{1}$  			&&&&& $92$ &&& $\mathbf{44}$ & $+$ & $44$ &&& $\mathbf{2}$	  		&&&&&   $^*144$ &&& $\mathbf{552}$ & $+$ & $48$ &&& $\mathbf{96}$       \\
$39$ &&& $\mathbf{24}$ & $+$ & $24$ &&& $\mathbf{1}$			&&&&&  $93$ &&& $\mathbf{60}$ & $+$ & $60$ &&& $\mathbf{1}$  		&&&&& $146$ &&& $\mathbf{72}$ & $+$ & $72$ &&& $\mathbf{1}$       \\
$40$ &&& $\mathbf{44}$ & $+$ & $16$ &&& $\mathbf{9}$	 		&&&&& $94$ &&& $\mathbf{46}$ & $+$ & $46$ &&& $\mathbf{1}$	  		&&&&&   $^*147$ &&& $\mathbf{960}$ & $+$ & $84$ &&& $\mathbf{68}$     \\ 
$42$ &&& $\mathbf{52}$ & $+$ & $12$ &&& $\mathbf{7}$  			&&&&& $^*96$ &&& $\mathbf{272}$ & $+$ & $32$ &&& $\mathbf{58}$   		&&&&&  $148$ &&& $\mathbf{144}$ & $+$ & $72$ &&& $\mathbf{3}$      \\
$44$ &&& $\mathbf{20}$ & $+$ & $20$ &&& $\mathbf{2}$   		&&&&&   $^*98$ &&& $\mathbf{480}$ & $+$ & $42$ &&& $\mathbf{38}$   		&&&&&  $^*150$ &&& $\mathbf{648}$ & $+$ & $40$ &&& $\mathbf{74}$       \\  \addlinespace[.2em]
$^*45$ &&& $\mathbf{16}$ & $+$ & $24$ &&& $\mathbf{8}$ 			&&&&&  $^*99$ &&& $\mathbf{40}$ & $+$ & $60$ &&& $\mathbf{20}$  		&&&&&  $152$ &&& $\mathbf{144}$ & $+$ & $72$ &&& $\mathbf{23}$       \\
$46$ &&& $\mathbf{22}$ & $+$ & $22$ &&& $\mathbf{1}$		  	&&&&&  $^*100$ &&& $\mathbf{512}$ & $+$ & $40$ &&& $\mathbf{42}$  		&&&&&  $^*153$ &&& $\mathbf{64}$ & $+$ & $96$ &&& $\mathbf{32}$       \\
$48$ &&& $\mathbf{64}$ & $+$ & $16$ &&& $\mathbf{20}$			&&&&&   $102$ &&& $\mathbf{96}$ & $+$ & $32$ &&& $\mathbf{19}$  		&&&&&  $154$ &&& $\mathbf{180}$ & $+$ & $60$ &&& $\mathbf{17}$     \\ 
$^*49$ &&& $\mathbf{180}$ & $+$ & $42$ &&& $\mathbf{30}$  		&&&&&  $104$ &&& $\mathbf{132}$ & $+$ & $48$ &&& $\mathbf{13}$   	&&&&&  $155$ &&& $\mathbf{120}$ & $+$ & $120$ &&& $\mathbf{1}$      \\
$^*50$ &&& $\mathbf{152}$ & $+$ & $20$ &&& $\mathbf{18}$    		&&&&&  $105$ &&& $\mathbf{48}$ & $+$ & $48$ &&& $\mathbf{4}$   		&&&&&  $156$ &&& $\mathbf{352}$ & $+$ & $48$ &&& $\mathbf{22}$       \\  \addlinespace[.2em]
$52$ &&& $\mathbf{48}$ & $+$ & $24$ &&& $\mathbf{3}$  			&&&&&   $106$ &&& $\mathbf{52}$ & $+$ & $52$ &&& $\mathbf{1}$  		&&&&&  $158$ &&& $\mathbf{78}$ & $+$ & $78$ &&& $\mathbf{1}$       \\
$^*54$ &&& $\mathbf{246}$ & $+$ & $18$ &&& $\mathbf{33}$		&&&&&   $^*108$ &&& $\mathbf{492}$ & $+$ & $36$ &&& $\mathbf{66}$  	&&&&&  $^*160$ &&& $\mathbf{616}$ & $+$ & $64$ &&& $\mathbf{84}$       \\
$55$ &&& $\mathbf{40}$ & $+$ & $40$ &&& $\mathbf{1}$	  		&&&&&   $110$ &&& $\mathbf{168}$ & $+$ & $40$ &&& $\mathbf{9}$  		     \\ 
$56$ &&& $\mathbf{48}$ & $+$ & $24$ &&& $\mathbf{11}$			&&&&&  $111$ &&& $\mathbf{72}$ & $+$ & $72$ &&& $\mathbf{1}$   	      \\
$57$ &&& $\mathbf{36}$ & $+$ & $36$ &&& $\mathbf{1}$	  		&&&&&  $112$ &&& $\mathbf{192}$ & $+$ & $48$ &&& $\mathbf{36}$   		       \\  \bottomrule\addlinespace[.3em]
\end{tabular}
\caption{Skew morphisms of cyclic groups of order $n$}\label{tab:cyclic}
\end{table}

\subsection{Remarks concerning Table~\ref{tab:cyclic}}\label{subsec:tabrem}

An inspection of Table~\ref{tab:cyclic} suggests various interesting questions regarding skew morphisms of cyclic groups. Possibly the most natural question to ask here is which values $n$ actually appear in Table~\ref{tab:cyclic} or, equivalently, which cyclic groups admit a proper skew morphism. This was answered in~\cite{KovacsNedela2011} for cyclic groups, and later in~\cite{ConderJajcayTucker2016} for all other abelian groups. Specifically, if an abelian group $A$ does not admit any proper skew morphism, then $A$ is cyclic of order $n$ where $n=4$ or $\gcd(n,\phi(n))=1$, or $A$ is an elementary abelian $2$-group.

Next, we look at values $n$ such that $\ZZ_n$ admits (up to conjugacy in $\Aut(\ZZ_n)$) only one proper skew morphism. In~\cite{KovacsNedela2011} this was shown to be true for all cases where $n$ is a product of two distinct primes and $\gcd(n,\phi(n))>1$. The only other value $n$ that appears in Table~\ref{tab:cyclic} and has this property is $n=8$. An interesting question raised in this context is whether this covers all such values $n$, or if there are others.

We are also interested in those cyclic groups that admit only coset-preserving skew morphisms. Unlike general skew morphisms, coset-preserving skew morphisms are well understood for cyclic groups, and, in particular, we can list all coset-preserving skew morphisms of $\ZZ_n$ in polynomial time; see~\cite{BachratyJajcay2017}. Thus, if for some $n$ we can show that all skew morphisms of $\ZZ_n$ are coset-preserving, then we can find all skew morphisms of $\ZZ_n$ much faster than by using the algorithm explained in Section~\ref{subsec:alg}. In the following section we completely solve this problem by characterising all cyclic groups that admit only coset-preserving skew morphisms.

\section{Cyclic groups that admit only coset-preserving skew morphisms}\label{sec:cospres}

In this section we focus on cyclic groups admitting only coset-preserving skew morphisms. Our main theorem is the following:

\begin{theorem}\label{thm:main}
All skew morphisms of $\ZZ_n$ are coset-preserving if and only if $n=2^e m$ with $e\in \{0,1,2,3,4\}$ and $m$ odd and square-free.
\end{theorem}

Note that Theorem~\ref{thm:main} include all groups $\ZZ_n$ that does not admit any proper skew morphism, as in that case either $n=4$, or $(n,\phi(n))=1$, which forces $n$ to be square-free (for if some prime square $p^2$ divides $n$, then $p$ is a common factor of $n$ and $\phi(n)$). In what follows, we will say that the positive integer $n$ is \emph{resolvable} if it is expressible in the form $n=2^e m$ with $e\in \{0,1,2,3,4\}$ and $m$ odd and square-free. The proof of Theorem~\ref{thm:main} is split into two parts; in Section~\ref{subsec:proof1} we show that if $n$ is not resolvable, then $\ZZ_n$ admits a non-coset-preserving skew morphism, and in Section~\ref{subsec:proof2} we show that if $n$ is resolvable, then $\ZZ_n$ does not admit a non-coset-preserving skew morphism. Then in Section~\ref{subsec:enum} we use Theorem~\ref{thm:main} (and further facts about coset-preserving skew morphisms of cyclic groups) to enumerate all skew morphisms for many finite cyclic groups for which no such enumeration was available to date. Finally, in Section~\ref{subsec:4p} we give an example that demonstrates how Theorem~\ref{thm:main} can be applied to find a precise formula for the number of skew morphisms of $\ZZ_n$ in the case when $n$ is resolvable and has a relatively small number of prime factors.

\subsection{Cyclic groups admitting non-coset-preserving skew morphisms}\label{subsec:proof1}

Here we show that if the order of a cyclic group is divisible by $32$ or by the square of an odd prime, then this group admits a skew morphism that does not preserve the cosets of its kernel. To do this, we will use some facts about skew morphisms of $\ZZ_n$ that give rise to a regular Cayley map. First, we have the following: 

\begin{proposition}[\cite{JajcaySiran}]\label{prop:rise}
A skew morphism $\varphi$ of a finite group $B$ gives rise to a regular Cayley map for $B$ if and only if the set of elements of some orbit of $\langle\varphi\rangle$ is closed under taking inverses and generates $B$.
\end{proposition}

We say that a skew morphism $\varphi$ of a group $B$ is \emph{$t$-balanced} if its kernel has index $2$ in $B$. The value $t$ is given by $t=\pi(a)$, where $a$ is any element of $B$ not contained in $\ker\varphi$. In the special case when $t=\ord(\varphi)-1$ we say that $\varphi$ is anti-balanced. For further information on $t$-balanced skew morphisms we refer the reader to~\cite{ConderJajcayTucker2007}. The following observation shows that every coset-preserving skew morphism of $\ZZ_n$ that gives rise to a regular Cayley map is either an automorphism of $\ZZ_n$, or a $t$-balanced skew morphism of $\ZZ_n$.

\begin{lemma}\label{lem:rise}
If $\varphi$ is a coset-preserving skew morphism of $\ZZ_n$ that gives rise to a regular Cayley map, then the index of $\ker\varphi$ in $\ZZ_n$ is at most two. In particular, if $n$ is odd, then $\ker\varphi=\ZZ_n$ and $\varphi$ is an automorphism of $\ZZ_n$.
\end{lemma}
\begin{proof}
Since $\varphi$ gives rise to a regular Cayley map, by Proposition~\ref{prop:rise} there exists some orbit $T$ of $\langle\varphi\rangle$ that is closed under taking inverses and generates $\ZZ_n$. Further, by~\cite[Corollary 3.3]{KovacsNedela2011} we know that $T$ contains some element $t$ such that $\langle t \rangle = \ZZ_n$, and since $T=-T$, we also have $-t \in T$. Next, from the fact that $\varphi$ is coset-preserving we deduce that $t$ and $-t$ are both in the same coset of $\ker\varphi$ in $\ZZ_n$. It follows that $2t\in \ker\varphi$, and noting that $t$ is a generator of $\ZZ_n$, we also have $\gcd(n,t)=1$. Hence $2\in \ker\varphi$, and the rest follows.
\end{proof}

Throughout the proof of the following proposition we repeatedly refer to the classification of regular Cayley maps for cyclic groups given in~\cite{ConderTucker}.

\begin{proposition}\label{prop:product}
If a positive integer $n$ is divisible by $32$ or $p^2$ for some odd prime $p$, then $\ZZ_n$ admits a skew morphism that is not coset-preserving.
\end{proposition}
\begin{proof}
First, assume that $n$ is odd and divisible by $p^2$ for some odd prime $p$. Then there exists a regular Cayley map for $\ZZ_n$ with non-balanced representation (see~\cite[Section 8]{ConderTucker}), and hence there exists a proper skew morphism $\varphi$ of $\ZZ_n$ that gives rise to this Cayley map. Since $\varphi$ is proper and $n$ is odd, by Lemma~\ref{lem:rise} we deduce that $\varphi$ is not coset-preserving.

Next, if $n$ is even and divisible by $p^2$ for some odd prime $p$, then we have $\ZZ_n = \ZZ_\ell \times \ZZ_{2^e}$ with $\ell$ odd. Since $\ell$ is clearly divisible by $p^2$, from the previous paragraph we know that $\ZZ_\ell$ admits a skew morphism $\varphi$ that is not coset-preserving. By Lemma~\ref{lem:product} there exists a skew morphism $\theta$ of $\ZZ_n$ such that $\theta \restriction_{\ZZ_\ell} = \varphi$ and $\ker\theta = \ker\varphi \times \ZZ_{2^e}$. Now it can be easily seen that since $\varphi$ does not preserve the cosets of $\ker\varphi$ in $\ZZ_\ell$, the same is true for $\theta$ and cosets of $\ker\theta$ in $\ZZ_n$.

Finally, let $n$ be even and divisible by $32$, and consider the factorisation $\ZZ_n = \ZZ_{2^e} \times \ZZ_\ell$ with $\ell$ odd. Note that to show that $\ZZ_n$ admits a non-coset-preserving skew morphism, it is sufficient to prove this for $\ZZ_{2^e}$ (and the rest will follow by Lemma~\ref{lem:product}). Let $M(2m,r)$ be the regular Cayley map for $\ZZ_{2m}$ given by~\cite[Definition 3.6]{ConderTucker}. This map is defined for every unit $r$ modulo $m$ such that if $b$ is the largest divisor of $m$ that is relatively prime to $r-1$, then either $b=1$, or $r$ is a root of $-1$ modulo $b$ of multiplicative order $2k$ where $k$ is relatively prime to $m/b$. Let $m=2^{e-1}$, $r=2^{e-3}+1$, and $M=M(2m,r)$. Note that the largest divisor of $m$ relatively prime to $r-1$ is $1$, and hence $b=1$. Also note that $r$ is not a root of $-1$ modulo $m$, and since $e\geq 5$ we have $r^2 \not\equiv 1 \pmod{m}$. It follows that $M$ has no balanced, no $t$-balanced, and no anti-balanced representation; see~\cite[Section 8]{ConderTucker}. Since every automorphism of $\ZZ_{2^e}$ gives rise to a skew morphism with a balanced representation, and every skew morphism of $\ZZ_{2^e}$ with kernel of index $2$ in $\ZZ_{2^e}$ gives rise to a skew morphism with either $t$-balanced or anti-balanced representation, we deduce that a skew morphism $\varphi$ of $\ZZ_{2^e}$ that gives rise to $M$ has kernel of index greater than two in $\ZZ_{2^e}$. Hence by Lemma~\ref{lem:rise} we find that $\varphi$ is not coset-preserving.
\end{proof}

\subsection{Cyclic groups admitting only coset-preserving skew morphisms}\label{subsec:proof2}

Next we show that if the positive integer $n$ is resolvable, then all skew morphisms of $\ZZ_n$ are coset-preserving. We start with the following technical lemma.

\begin{lemma}\label{lem:tech}
Let $\varphi$ be a skew morphism of a cyclic group $\ZZ_n$, let $N$ be any non-trivial subgroup of $\ker \varphi$, let $\varphi_N^*$ be the skew morphism of $\ZZ_n/N$ induced by $\varphi$, and let $L/N$ be the kernel of $\varphi_N^*$. Also let $s$ be a prime factor of $n/|\ker \varphi|$, let $k_s$ denote the largest power of $s$ that divides $|(\ker\varphi)/N|$, and let $a=n/(s|\ker\varphi|)$ be an element of $\ZZ_n$. If $sk_s$ divides $|L/N|$, then $a\notin \ker\varphi$ and $a\in L$.  
\end{lemma}
\begin{proof}
First, since $a|\ker\varphi|=n/s<n$, we have $a\notin \ker\varphi$. (Note that this part is true regardless of whether $sk_s$ divides $|L/N|$.) Next, let $K=\ker\varphi$, and let $m$ be an integer such that $|K/N|=mk_s$. (Observe that $m$ and $k_s$ are relatively prime.) Then since $K/N$ is a subgroup of $L/N$, we know that $mk_s$ divides $|L/N|$. But $|L/N|$ is also divisible by $sk_s$, and since $\gcd(m,k_s)=1$, it follows that $|L/N|$ must be divisible by $msk_s$. Hence $|L|$ is divisible by $s|K|$, and therefore $a\in L$.  
\end{proof}

We are now ready to prove the key part of the proof of Theorem~\ref{thm:main}.

\begin{proposition}\label{prop:main}
Let $n=2^em$ with $e\in \{0,1,2,3,4\}$ and $m$ odd and square-free. Then every skew morphism of $\ZZ_n$ is coset-preserving.
\end{proposition}
\begin{proof}
Suppose to the contrary that the assertion is not true, and let $n$ be the smallest resolvable integer such that $\ZZ_n$ admits a skew morphism $\varphi$ that is not coset-preserving. Also let $K=\ker\varphi$, let $\varphi^*$ denote the skew morphism of $\ZZ_n/K$ induced by $\varphi$, and let $\overline{\varphi}$ be a quotient of $\varphi$. Since the only skew morphism of the trivial group is clearly coset-preserving, we have $n>1$. Recalling that every skew morphism of a non-trivial group has a non-trivial kernel, it follows that $|K|\geq 2$. In particular, it follows that $|K|$ has at least one prime factor. We proceed by considering the two following cases: 

\medskip\noindent
\textbf{Case (a)}: $|K|$ is a prime power
\medskip

Let $p$ be the largest prime divisor of $n$. Then by Corollary~\ref{cor:order kernel} we know that $p$ divides $|K|$. If $p=2$, then $n=2$, $4$, $8$, or $16$, in which case $\ZZ_n$ does not admit a skew morphism that is not coset-preserving; see \cite{ConderList}. Hence $p$ is odd and $|K|=p$, and thus $|\ZZ_n/K|=n/p$. Then since $p$ is the largest prime divisor of $n$, we know that $p$ does not divide $|\ZZ_n/K|\phi(|\ZZ_n/K|)$, and it follows from Theorem~\ref{thm:orderofskewcyclic} that $p$ does not divide the order of $\varphi^*$. Further, by Lemma~\ref{lem:tech(b)} we know that $p$ divides $\ord(\varphi)$, and since $\ord(\varphi^*)=p_{\varphi}$, we deduce that $p$ divides $\ord(\varphi^{p_{\varphi}})$. On the other hand, noting that $\varphi^{p_{\varphi}}$ preserves the cosets of $K$ in $\ZZ_n$, we have $\ord(\varphi^{p_{\varphi}})\leq |K| = p$, and hence $\ord(\varphi^{p_{\varphi}}) = p$. Therefore $\ord(\varphi)=pp_{\varphi}$, and then by Proposition~\ref{prop:quotientinv:b} we have $|\ker\overline{\varphi}|=p$. But $p$ does not divide $(n/p)$, which (by Lemma~\ref{lem:quotientmain:a}) is the order of $\overline{\varphi}$, and so by Lemma~\ref{lem:tech(b)} we find that $\overline{\varphi}$ is a group automorphism. Hence by Lemma~\ref{lem:quotientcp} we deduce that $\varphi$ is coset-preserving, contradiction.

\medskip\noindent
\textbf{Case (b)}: $|K|$ has at least two distinct prime factors
\medskip

Let $k=|K|$, and let $d$ denote the integer satisfying $n=kd$. (Note that the elements of $K$ are exactly the multiples of $d$ modulo $n$.) Also let $d=r_1\dots r_{\ell}$ be a factorisation such that each factor is either an odd prime or the maximum possible power of two. Since $\varphi$ does not preserve the cosets of $K$ in $\ZZ_n$, we know that $\varphi(1)\not\equiv 1 \pmod{d}$. Hence, noting that all factors $r_i$ for $i\in \{1,\dots,\ell\}$ are pairwise relatively prime, it follows from the Chinese Remainder Theorem that there exists $r\in \{r_1,\dots,r_\ell\}$ such that $\varphi(1)\not\equiv 1 \pmod{r}$. We will show that this cannot happen.

First, assume that $r$ is odd. It follows that $r$ is a prime, and also that $n$ is not divisible by $r^2$. (And so, in particular, $r$ does not divide $|K|$.) Let $p$ be any prime factor of $|K|$, let $N$ be the unique subgroup of $K$ of order $p$, and let $\varphi^*_N$ be the skew morphism of $\ZZ_n/N$ induced by $\varphi$. Also let $L/N$ be the kernel of $\varphi^*_N$, and suppose that $|L/N|$ is not divisible by $r$. Since the order of the cyclic group $\ZZ_n/N$ is clearly resolvable, by the assumption of minimality of $n$ we know that $\varphi^*_N$ is coset-preserving. Then since $r$ divides $\ZZ_n/N$ but not $L/N$, we have $\varphi(1) \equiv \varphi^*_N(1)\equiv 1 \pmod{r}$. But this contradicts the fact that $\varphi(1)\not\equiv 1 \pmod{r}$, and hence we deduce that $r$ divides $|L/N|$. Now, since $r$ does not divide $|K|$, it follows that $r$ does not divide $|K/N|$, and thus we may use Lemma~\ref{lem:tech} (with $s=r$ and $k_s=1$). Hence we deduce that the element $a=d/r$ of $\ZZ_n$ is not contained in $K$, but also $a\in L$, and by Lemma~\ref{lem:tech(a)} it follows that $p\pi(a)\equiv p \pmod{\ord(\varphi)}$. Since $p$ was an arbitrary prime factor of $|K|$, the same is true also for some other prime factor $q$ of $|K|$. (Here we use the assumption that the order of $K$ has at least two distinct prime factors.) Hence we have
\begin{align*}
 p\pi(a)&\equiv p \pmod{\ord(\varphi)} , \\
 q\pi(a)&\equiv q \pmod{\ord(\varphi)} ,  
\end{align*}
for a pair of distinct primes $p$ and $q$. Then since $\gcd(p,q)=1$, we deduce that $\pi(a) \equiv 1 \pmod{\ord(\varphi)}$, and consequently $a\in K$, contradiction.

Next, assume that $r$ is even. Again let $p$ denote a prime factor of $|K|$, and define $N$, $\varphi^*_N$, and $L/N$ same as in the case when $r$ was odd. Also let $k_2$ denote the largest power of $2$ that divides $|K/N|$, and suppose that $|L/N|$ is not divisible by $2k_2$. Noting that $K/N$ is a subgroup of $L/N$, it follows that the largest power of $2$ that divides $|L/N|$ must be $k_2$. Then since $\varphi^*_N$ must be coset-preserving (due to the minimality of $n$) and the largest power of $2$ that divides $|\ZZ_n/N|$ is equal to $rk_2$, we deduce that $\varphi(1) \equiv \varphi^*_N(1)\equiv 1 \pmod{r}$, contradicting the fact that $\varphi(1)\not\equiv 1 \pmod{r}$. Hence it follows that $2k_2$ divides $|L/N|$, and Lemma~\ref{lem:tech} (in this case we take $s=2$ and $k_s=k_2$) implies that the element $a=d/2$ of $\ZZ_n$ satisfies $a\notin K$ and $a\in L$. Using the same argument as for $r$ odd this again leads to a contradiction.
\end{proof}

Theorem~\ref{thm:main} now follows directly from Propositions~\ref{prop:product} and~\ref{prop:main}.

\subsection{Enumeration}\label{subsec:enum}

In~\cite{BachratyJajcay2017} it was shown that each coset-preserving skew morphism $\varphi$ of $\ZZ_n$ is uniquely determined by the following four parameters: the smallest non-zero element $d$ of $\ker\varphi$; the element $h$ of $\ZZ_n$ such that $\varphi(1)=1+h$; the smallest positive integer $s$ such that $\varphi(d)=sd$; and the positive integer $e=\pi(1)$. (Note that $s$ always exists since $d\in \ker\varphi$ and $\varphi$ restricts to an automorphism of $\ker\varphi$.) Using various properties of coset-preserving skew morphisms it can be checked that if $\varphi$ is non-trivial, then the parameters $d$, $h$, $s$ and $e$ must satisfy the following properties (see~\cite[Section 4]{BachratyJajcay2017} for details):

\begin{enumerate}[label={\rm (\roman*)}]
\item\label{item:1} all four parameters are positive integers; 
\item\label{item:2} $d$ is a proper divisor of $n$;
\item\label{item:3} $s<n/d$ and $\gcd(s,n/d)=1$;
\item\label{item:4} $h$ is a multiple of $d$ strictly smaller than $n$;
\item\label{item:5} if $r$ is the smallest positive integer such that $h\sum_{i=0}^{r-1}  s^i \equiv 0 \pmod{n}$,\\ then $e$ is a (multiplicative) unit modulo $r$ of order $d$ and $e<r$; 
\item\label{item:6} $sd\equiv \sum\limits_{j=0}^{d-1}(1+h\sum\limits_{i=0}^{\ell_j} s^i) \pmod{n}$, where $\ell_j = e^j-1 \bmod{r}$; and
\item\label{item:7} $s^{e-1} \equiv 1 \pmod{n/d}$.
\end{enumerate}

On the other hand, for each set of parameters $(d,h,s,e)$ satisfying all of the above properties there exists a unique non-trivial coset-preserving skew morphism of $\ZZ_n$ (which can be constructed in a straightforward way) with this parameter set; see~\cite[Section 5]{BachratyJajcay2017}. This gives a one-to-one correspondence between non-trivial coset-preserving skew morphisms of $\ZZ_n$ and the sets of parameters $(d,h,s,e)$, and this correspondence can be used to find all coset-preserving skew morphisms of a given cyclic group in a polynomial time (in the cardinality of the group). Hence, by Theorem~\ref{thm:main} we can quickly find all skew morphisms of $\ZZ_n$, where $n$ is resolvable. Using the above observations, we developed an algorithm that can enumerate all skew morphisms for any given cyclic group of any resolvable order. A \CC\ implementation of this algorithm succeeded in enumerating all skew morphisms for cyclic groups of resolvable orders smaller than $10000$ within a second, even without parallelisation. In comparison, the best available method to date for finding all skew morphisms for cyclic groups of general order (described in Section~\ref{subsec:alg}) is computationally feasible only up to order $161$. In our enumeration, which is available at~\cite{Enum}, we provide the total number of skew morphisms of $\ZZ_n$, and also the total number of automorphisms and their proportion among all skew morphisms.

\subsection{Skew morphisms of $\ZZ_{4p}$}\label{subsec:4p}

Although coset-preserving skew morphisms can be generated efficiently, there is no known explicit formula for the number of coset-preserving skew morphisms of $\ZZ_n$ for general $n$. Such a formula would be useful, and in the case when $n$ is resolvable it would give the number of all skew morphisms of $\ZZ_n$. Resolvable integers $n$ with the simplest structure (with respect to their prime factorisation) are primes, in which case all skew morphisms of $\ZZ_n$ are automorphisms of $\ZZ_n$. The situation is also completely understood in the case when $n$ is a product of two distinct primes $p$ and $q$, in which case the number of skew morphisms of $\ZZ_{pq}$ is $(p-1)(q-1)$ if $\gcd(p,q-1)=\gcd(p-1,q)=1$, and $2(p-1)(q-1)$ otherwise; see~\cite{KovacsNedela2011} for example. Here we go one step further and find a formula for the number of skew morphisms of $\ZZ_{4p}$, where $p$ is an odd prime. We will use the following fact:

\begin{proposition}\label{prop:orders 4p}
Let $p$ be an odd prime. If $\varphi$ is a proper skew morphism of $\ZZ_{4p}$, then the action of $\varphi$ on its kernel is trivial.
\end{proposition}
\begin{proof}
Let $\pi$ and $K$ denote the power function and the kernel of $\varphi$, and let $\varphi^*$ be the skew morphism of $\ZZ_n/K$ induced by $\varphi$. Note that by Theorem~\ref{thm:main} we know that $\varphi$ is coset-preserving, and so $\varphi^*$ must be trivial. Further, by Corollary~\ref{cor:order kernel} we know that $p$ divides $|K|$, and since $\varphi$ is proper it follows that the order of $K$ is either $p$ or $2p$. We proceed by considering these two cases. In each case, we let $T$ be the orbit of $\langle \varphi \rangle$ that contains $1$. Note that by Proposition~\ref{prop:skewgenorbit} we have $|T|=\ord(\varphi)$.

\medskip\noindent
\textbf{Case (a)}: $|K| = p$
\medskip

Since $\varphi$ is coset-preserving, we know that the cosets of $K$ in $\ZZ_n$ are preserved set-wise by $\varphi$. Note that only one element of the coset $1+K$ does not generate $\ZZ_n$ (either $p$ or $3p$, depending on whether $p \equiv 1 \pmod{4}$ or $p \equiv 3 \pmod{4}$), and since by Lemma~\ref{lem:fix} no elements outside of $K$ are fixed by $\varphi$, it follows that each orbit of $\langle \varphi \rangle$ on $1+K$ generates $\ZZ_n$. Hence by Proposition~\ref{prop:skewgenorbit} we know that all of these orbits have size $\ord(\varphi)$, and since $|1+K| = |K| = p$, we see that $\ord(\varphi) = p$. Since $\varphi$ restricts to an automorphism of $K$, and $\ord(\varphi)=|K|=p$, we deduce that the action of $\varphi$ on $K$ is trivial.
 
\medskip\noindent
\textbf{Case (b)}: $|K| = 2p$
\medskip

In this case, since $\varphi$ is proper, by Theorem~\ref{thm:orderofskewcyclic} we have $\gcd(\ord(\varphi),4p)>1$. If the order of $\varphi$ is odd, then we have $\ord(\varphi)=p$ or $\ord(\varphi)=3p$, but the latter case can be easily excluded since $\varphi$ must preserve both cosets of $K$ in $\ZZ_n$ of size $2p$. 

Next we deal with the case when $\ord(\varphi)$ is even. Note that by Lemma~\ref{lem:quotientmain:a} we have $\pi(1)=\overline{\varphi}(1)$, and since $\overline{\varphi}$ is an automorphism of the cyclic group $\ZZ_{\ord(\varphi)}$ of even order, we deduce that $\pi(1)$ is odd. We will use this observation to show that both $p$ and $3p$ are contained in some orbits of $\langle \varphi \rangle$ that generate $\ZZ_n$. Suppose to the contrary that this is not true. Since every element of $1+K$ other than $p$ and $3p$ generates $\ZZ_n$ and no element of $1+K$ is fixed by $\varphi$, we must have $\varphi(p) = 3p$ and $\varphi(3p) = p$. Then since $\pi(1)$ is odd, we have $\varphi^{\pi(1)}(p) = 3p$, and therefore $\varphi(1+p) = \varphi(1)+\varphi^{\pi(1)}(p) = \varphi(1)+3p$. On the other hand, we have $\varphi(p+1) = \varphi(p)+\varphi^{\pi(p)}(1) = 3p+\varphi^{\pi(p)}(1)$. But then $\varphi(1)=\varphi^{\pi(p)}(1)$, and since $|T|=\ord(\varphi)$ we find that $\pi(p)=1$. This forces $p \in K$, contradicting the fact that the order of $K$ is $2p$. Hence we deduce that all orbits of $\langle \varphi \rangle$ on $1+K$ generate $\ZZ_n$. In particular, $\ord(\varphi)$ divides $2p$, and it follows that $\ord(\varphi) = 2p$. 

We have shown that if $|K| = 2p$, then $\ord(\varphi)$ is equal to $p$ or $2p$. Hence by order considerations it can be easily seen that $\varphi$ acts on $K$ either trivially, or by the inversion. To exclude the latter, first note that $1$ and $\varphi(1)$ are in the same coset of $K$ in $\ZZ_{4p}$, and hence $\varphi(1)-1 \in K$. If we let $h=\varphi(1)-1$, then we have $\varphi^2(1) = \varphi(h+1)=\varphi(h)+\varphi(1) = -h + h + 1= 1$, but then by Proposition~\ref{prop:skewgenorbit} we have $\ord(\varphi)=2$, which is impossible. Hence we again conclude that the action of $\varphi$ on $K$ is trivial.
\end{proof}

Using Theorem~\ref{thm:main} and Proposition~\ref{prop:orders 4p} we can now easily enumerate all proper skew morphisms of $\ZZ_{4p}$.

\begin{theorem}
If $p$ is an odd prime, then the number of skew morphisms of $\ZZ_{4p}$ is
\[
    \begin{cases}
      6p-6 & \text{if}\ p\equiv 1\ (\bmod{\, 4})\\
      4p-4 & \text{if}\ p\equiv 3\ (\bmod{\, 4}).
    \end{cases}  
\]
\end{theorem}
\begin{proof}
Throughout this proof we refer to the properties \ref{item:1} to \ref{item:7} and the parameters $d$, $h$, $s$ and $e$ of coset-preserving skew morphisms for cyclic groups explained in Section~\ref{subsec:enum}. We know that $|\Aut(\ZZ_{4p})|=2p-2$, so we proceed by counting proper skew morphisms of $\ZZ_{4p}$. Let $\varphi$ be a proper skew morphism of $\ZZ_{4p}$, and recall that by Theorem~\ref{thm:main} it is coset-preserving. Let $d$, $h$, $s$ and $e$ be the four defining parameters of $\varphi$, and note that by Proposition~\ref{prop:orders 4p} we have $s=1$. Since $p$ is the largest prime divisor of $4p$, by Corollary~\ref{cor:order kernel} we know that $p$ divides $|\ker\varphi|$, and it follows that $d=2$ or $d=4$. (The case $d=1$ can be excluded as $\varphi$ is proper.)

First let $d=2$, and let $h$ be any positive multiple of $2$ strictly smaller than $4p$. If $4$ divides $h$, then by \ref{item:5} we find that $r=p$ and $e=p-1$. Since $s=1$, both \ref{item:3} and \ref{item:7} are trivially true, and \ref{item:6} holds as $(1+h)+(1+h(p-1))\equiv 2 + hp \equiv 2 \pmod{4p}$. If $4$ does not divide $h$ and $h\neq 2p$, then by \ref{item:5} we have $r=2p$ and $e=2p-1$. Again both \ref{item:3} and \ref{item:7} hold trivially, and \ref{item:6} is also true since $(1+h)+(1+h(2p-1) \equiv 2 +2hp \equiv 2 \pmod{4p}$. If $h=2p$, then it can be easily verified that $\ord(\varphi)=r=2$, contradicting the fact that $\varphi$ is proper. Since for all but one choice of $h$ we obtain exactly one coset-preserving skew morphism, it follows that in this case we have exactly $2p-2$ skew morphism of $\ZZ_{4p}$. 

Next let $d=4$, and let $h$ be any positive multiple of $4$ strictly smaller than $4p$. Then by \ref{item:5} we deduce that $r=p$, and $e$ must be a fourth root of unity modulo $p$. It follows that necessarily $p\equiv 1 \pmod{4}$, in which case there are two possible candidates for $e$. Note that for either candidate we have $e^2\equiv -1 \pmod{p}$ and $e^3\equiv -e \pmod{p}$. Again \ref{item:3} and \ref{item:7} hold trivially, and \ref{item:6} holds as well since $(1+h)+(1+he)+(1+h(p-1))+(1+h(p-e))\equiv 4 + 2hp \equiv 4 \pmod{4p}$. Hence, if $p\equiv 1 \pmod{4}$, then every choice of $h$ gives two coset-preserving skew morphisms (one for each choice of $e$), which gives a total of $2p-2$ skew morphisms of $\ZZ_{4p}$.
\end{proof}

\nocite{*}
\bibliographystyle{amsplain-my}
\bibliography{quotients}
\end{document}